\documentclass[a4paper]{amsart}
\usepackage{amsmath,amsfonts,amssymb,latexsym}
\usepackage{hyperref}
\usepackage{graphicx}

\newcommand{\RR}{\mathbb{R}}
\newcommand{\CC}{\mathbb{C}}
\newcommand{\NN}{\mathbb{N}}
\newcommand{\ZZ}{\mathbb{Z}}
\newcommand{\QQ}{\mathbb{Q}}
\newcommand{\OO}{\mathcal{O}}
\newcommand{\Ss}{\mathcal{S}}

\newtheorem{Th}{Theorem}

\newtheorem{Pro}{Proposition}
\newtheorem{Ex}{Example}

\theoremstyle{definition}
\newtheorem{Df}{Definition}

\theoremstyle{remark}
\newtheorem{Rem}{Remark}

\begin{document}
\subjclass[2010]{}
\title[The Stokes phenomenon]{The Stokes phenomenon for certain PDEs in a case when initial data have a
finite set of singular points}
\author{Bo\.{z}ena Tkacz}
\address{Faculty of Mathematics and Natural Sciences,
College of Science\\
Cardinal Stefan Wyszy\'nski University\\
W\'oycickiego 1/3,
01-938 Warszawa, Poland}
\email{bpodhajecka@o2.pl}

\begin{abstract}
We study the Stokes phenomenon via hyperfunctions for the solutions of the 1-dimensional complex heat equation
 under the condition that the Cauchy data are holomorphic on $\mathbb{C}$ but a finitely many singular or branching points with the appropriate growth condition at the infinity. The main tool are the theory of summability and the theory of hyperfunctions, which allows us to describe jumps across Stokes lines.

\end{abstract}

\maketitle

\section{Introduction}

This paper deals with the 1-dimensional complex heat equation  $\partial_t u(t,z)=\partial_z^2 u(t,z)$, $u(0,z)=\varphi(z)$. The aim of this work is to
describe jumps across the Stokes lines in terms of hyperfunctions in the case when the initial data $\varphi(z)$ have a finite set of singular points. 
First, we consider the function $\varphi(z)$ which has a single-valued singular point and we derive the jump in a form of convergent series (see Theorem \ref{tw:1}). Then we discuss the case when the function $\varphi(z)$ has a multi-valued singular point and we give the integral representation of the jump (see Theorem \ref{tw:3}). Thus we obtain a full characterization of the Stokes phenomenon for the considered equation.
 At the end, we extend our results to the generalization of the heat equation.\\
 The important point to note here is that  D.A.~Lutz, M.~Miyake and R.~Sch\"afke in \cite{L-M-S} considered the similar problem for the heat equation when the Cauchy data is a function $\varphi(z)=1/z$ with singularity at 0. They proved that the heat kernel was given by a function as a jump of Borel sum (see [ \cite{L-M-S}, Theorem 5.1]).\\
 It is worth pointing out that this work is a continuation of the paper \cite{Mic-Pod} in which we study the heat equation with the Cauchy data given by a meromorphic function with a simple pole or finitely many poles.

\section{Notation. Gevrey's asymptotics and $k$-summability}
In the paper we use the following notation.\\
A set of the form
\begin{enumerate}
\item \begin{displaymath}
S=S_d(\alpha,R)=\{z\in\tilde{\CC}\colon\ z=r\*e^{i\phi},\ r\in(0,R),\ \phi\in(d-\alpha/2,d+\alpha/2)\}.
\end{displaymath}
defines a sector $S$ in a direction $d\in\RR$ with an opening $\alpha>0$ and a radius $R\in\RR_+$ in the universal covering space $\tilde{\CC}$ of
$\CC\setminus\{0\}$,
\item \begin{displaymath}
 D_r=\{z\in\CC:|z|<r\}.
 \end{displaymath}
  defines a complex disc $D_r$ in $\CC$ with a radius $r>0$.
\end{enumerate}
In the case that\\
1. $R=+\infty$, then this sector is called unbounded and one can write $S=S_d(\alpha)$ for short,\\
2. the opening $\alpha$ is not essential, then the sector $S_d(\alpha)$ is denoted briefly by $S_d$,\\
3. the radius $r$ is not essential, the set $D_r$ will be designate by $D$.\\
 To simplify the notation, we abbreviate a set $S_d(\alpha)\cup D$
 (resp. $S_d\cup D$) to $\widehat{S}_d(\alpha)$ (resp. $\widehat{S}_d$).
 
If $f$ is a holomorphic function on a domain $G\subset\CC^n$, then it will be written as $f\in\OO(G)$.

The set of all formal power series (i.e. a power series $\sum_{n=0}^{\infty} a_{n}\*t^{n}$ created for a sequence of complex numbers $(a_n)_{n=0}^{\infty}$) will be represented by the symbol $\CC[[t]]$. Similarly, $\OO(D_r)[[t]]$  stands for the set of all formal power series 
 $\sum_{n=0}^{\infty} a_{n}(z)t^{n}$ with $a_{n}(z)\in\OO(D_r)$ for all $n\in\NN_{0}$.

\begin{Df}
Assume that $k>0$ and $f\in\OO(S)$. The function $f$ is called of \textit{exponential growth of order at most $k$}, if for every proper subsector $S^*\prec S$ (i.e. $\overline{S^{*}}\setminus\{0\} \subseteq S$) there exist constants $C_1, C_2>0$ such that
$|f(x)|\le C_1\*e^{C_2|x|^{k}}$ for every $x\in S^*$. \\
If the function $f$ is of exponential growth of order at most $k$, then one can write $f\in\OO^{k}(S)$.
\end{Df}

\begin{Df}
A power series $\sum_{n=0}^{\infty} a_{n}\*t^{n} \in\CC[[t]]$ is called a \emph{formal power series of Gevrey order $s$} ($s\in\RR$), if there exist positive constants $A, B>0$ such that $|a_n|\le A\*B^n\*(n!)^s$ for every $n\in\NN_0$. The set of all such formal power series is denoted by $\CC[[t]]_s$ (resp. $\OO(D_r)[[t]]_s$).
\end{Df}

\begin{Rem}
\label{prop:3}(see \cite {B2})
If $k<0$ then $u\in\CC[[t]]_k\Longleftrightarrow$ $u$ is convergent  and $u\in\OO^{-\frac{1}{k}}(\CC)$.
\end{Rem}

\begin{Df}
Assume that $s\in\RR$, $S$ is a given sector in $\tilde{\CC}$ and $f\in\OO(S)$. A power series
$\hat{f}(t)=\sum_{n=0}^{\infty} a_{n}\*t^{n}\in\CC[[t]]_{s}$
 is called \textit{Gevrey's asymptotic expansion of order $s$} of the function $f$ in $S$ (in symbols $f(t)\sim_s\hat{f}(t)$ in $S$) if  for every $S^*\prec S$ there exist positive constants $A, B>0$ such that for every $N\in\NN_0$ and every $t\in S^*$
 $$|f(t)-\sum_{n=0}^{N} a_{n}\*t^{n}|\le A\*B^N\*(N!)^{s}\*|t|^{N+1}.$$
\end{Df}

\label{sect}
To introduce the notion of summability, by Balser's theory of general moment summability (\cite[Section 6.5]{B2},
in particular \cite[Theorem 38]{B2}), we may take Ecalle's acceleration and deceleration operators instead of the standard Laplace and Borel
transform. 
\begin{Df}[see {\cite[Section 11.1]{B2}}]
Let $d\in\RR$, $\tilde{k}>\bar{k}>0$ and $k:=(1/\bar{k}-1/\tilde{k})^{-1}$.

The \emph{acceleration operator in a direction $d$ with indices $\tilde{k}$ and $\bar{k}$},
denoted by $\mathcal{A}_{\tilde{k},\bar{k},d}$, is defined for every $g(t)\in\OO^k(\widehat{S}_d)$ by
\begin{equation*}
 (\mathcal{A}_{\tilde{k},\bar{k},d}g)(t):=t^{-\bar{k}}\int_{e^{id}\RR_+}g(s)C_{\tilde{k}/\bar{k}}\big((s/t)^{\bar{k}}\big)\,ds^{\bar{k}},
\end{equation*}
where the \emph{Ecalle kernel $C_{\alpha}$} is defined by
\begin{gather}
\label{eq:ecalle_kernel}
 C_{\alpha}(\tau):=\sum_{n=0}^{\infty}\frac{(-\tau)^n}{n!\,\*\Gamma\bigl(1-\frac{n+1}{\alpha}\bigr)}\quad\text{for}\quad \alpha>1\,\,\, 
\end{gather}
and the integration is taken over the ray $e^{id}\RR_+:=\{re^{id}\colon r\geq 0\}$.

The \emph{formal deceleration operator with indices $\tilde{k}$ and $\bar{k}$}, denoted by $\hat{\mathcal{D}}_{\tilde{k},\bar{k}}$, is defined
for every $\hat{f}(t)=\sum_{n=0}^{\infty}a_{n}t^{n}\in\CC[[t]]$ by
\begin{equation*}
 (\hat{\mathcal{D}}_{\tilde{k},\bar{k}}\hat{f})(t):=\sum_{n=0}^{\infty}a_nt^n\frac{\Gamma(1+n/\tilde{k})}{\Gamma(1+n/\bar{k})}.
\end{equation*}
\end{Df}

\begin{Df}
Let $k>0$ and $d\in\RR$. A formal power series $\hat{f}(t)=\sum_{n=0}^{\infty} a_{n}t^{n}\in\CC[[t]]$ is
called $k$-\emph{summable in a direction} $d$ if 
\begin{gather*}
g(t)=(\hat{\mathcal{D}}_{1,\frac{k}{k+1}}\hat{f})(t)=
\sum_{n=0}^{\infty}a_n\frac{\Gamma(1+n)}{\Gamma(1+\frac{n(k+1)}{k})}t^n\in\OO^k(\widehat{S}_d(\varepsilon)) \quad \textrm{for some} \quad 
\varepsilon>0.
\end{gather*}

Moreover, the \emph{$k$-sum of $\hat{f}(t)$ in the direction $d$} is given by
\begin{equation}
\label{eq:ecalle}
f^d(t)=\Ss_{k,d}\hat{f}(t):=(\mathcal{A}_{1,\frac{k}{k+1},\theta}\hat{\mathcal{D}}_{1,\frac{k}{k+1}}\hat{f})(t)\quad\textrm{with}\quad
\theta\in (d-\varepsilon/2,d+\varepsilon/2).
\end{equation}
\end{Df}

\begin{Df}
If $\hat{f}$ is $k$-summable in all directions $d$ but (after identification modulo $2\pi$) finitely many directions $d_1,\dots,d_n$ then
$\hat{f}$ is called \emph{$k$-summable} and $d_1,\dots,d_n$ are called \emph{singular directions of $\hat{f}$}.
\end{Df}

\section{The Stokes phenomenon and hyperfunctions}

\subsection{The Stokes phenomenon for $k$-summable formal power series}
Now let us recall the concept of the Stokes phenomenon \cite[Definition 7]{Mic-Pod}.
\begin{Df}
\label{df:stokes}
Assume that $\hat{f}\in\CC[[t]]_{1/k}$ (resp. $\hat{u}\in\OO(D)[[t]]_{1/k}$) is $k$-summable with finitely many singular directions $d_1,d_2,\dots,d_n$. Then for every $l=1,\dots,n$ a set $\mathcal{L}_{d_{l}}=\{t\in\tilde{\CC}\colon \arg t=d_{l}\}$ 
is called a \emph{Stokes line for $\hat{f}$ (resp. $\hat{u}$)}. Of course every such Stokes line $\mathcal{L}_{d_l}$ for $\hat{f}$ (resp. $\hat{u}$) determines so called \emph{anti-Stokes lines
 $\mathcal{L}_{d_l\pm\frac{\pi}{2k}}$ for $\hat{f}$ (resp. $\hat{u}$)}.

Moreover, if $d_l^+$ (resp. $d_l^-$) denotes a direction close to $d_l$ and greater (resp. less) than $d_l$, and let $f^{d_l^+}=\mathcal{S}_{k, d_l^+}\hat{f}$ (resp. $f^{d_l^-}=\mathcal{S}_{k, d_l^-}\hat{f}$)
then the difference $f^{d_l^+}- f^{d_l^-} $ is called a \emph{jump} for $\hat{f}$ across the Stokes line $\mathcal{L}_{d_l}$. Analogously
we define the jump for $\hat{u}$.
\end{Df}

\begin{Rem}
Let  $r(t):=f^{d_l^+}(t)- f^{d_l^-}(t)$ for all $t\in S=S_{d_l}(\frac{\pi}{k})$. Then $r(t)\sim_{1/k} 0$ on $S$.
\end{Rem}

\subsection{Laplace type hyperfunctions}
We will describe jumps across the Stokes lines in terms of hyperfunctions. The similar approach to the Stokes phenomenon one can find
in \cite{Im, Mal3, S-S}. For more information about the theory of hyperfunctions we refer the reader to \cite{Kaneko}.

We will consider the space 
$$\mathcal{H}^k(\mathcal{L}_d):=\OO^k(D\cup (S_d\setminus \mathcal{L}_d))\Big/\OO^k(\widehat{S}_d)$$
of Laplace type hyperfunctions supported by $\mathcal{L}_d$ with exponential growth of order $k$. It means that every hyperfunction
$G\in\mathcal{H}^k(\mathcal{L}_d)$ may be written as
\[
G(s)=[g(s)]_{d}=\{g(s)+h(s)\colon h(s)\in\OO^{k}(\widehat{S}_d)\}
\]
for some defining function $g(s)\in\OO^k(D\cup (S_d\setminus \mathcal{L}_d))$. 

By the K\"othe type theorem \cite{Kot} one can treat the hyperfunction $G=[g(s)]_d$ as the analytic functional defined by
\begin{gather}
\label{eq:kothe}
G(s)[\varphi(s)]:=\int_{\gamma_d}g(s)\varphi(s)\,ds\quad\textrm{for sufficiently small}\quad \varphi\in\OO^{-k}(\widehat{S}_d)
\end{gather}
with $\gamma_{d}$ being a path consisting of the half-lines from 
$e^{id^-}\infty$ to $0$ and from $0$ to $e^{id^+}\infty$, i.e.
$\gamma_{d}=-\gamma_{d^-}+\gamma_{d^+}$ with $\gamma_{d^{\pm}}=\mathcal{L}_{d^{\pm}}$.

\subsection{The description of jumps across the Stokes lines in terms of hyperfunctions}
Assume that $\hat{f}$ is $k$-summable and $d$ is a singular direction. 
By (\ref{eq:ecalle}) the jump for $\hat{f}$ across the Stokes line $\mathcal{L}_d$ is given by 
\[
  f^{d^+}(t)-f^{d^-}(t)=(\mathcal{A}_{1,\frac{k}{k+1},d^+}-\mathcal{A}_{1,\frac{k}{k+1},d^-})\hat{\mathcal{D}}_{1,\frac{k}{k+1}}\hat{f}(t).
\]

We will describe this jump in terms of hyperfunctions.
To this end, observe that we can treat
$g(t):=\hat{\mathcal{D}}_{1,\frac{k}{k+1}}\hat{f}(t)\in\OO^{k}(D\cup (S_{d}\setminus \mathcal{L}_{d}))$
as a defining function of the hyperfunction $G(s):=[g(s)]_{d}\in \mathcal{H}^{k}(\mathcal{L}_{d})$.

So, for sufficiently small $r>0$ and $t\in S_{d}(\frac{\pi}{k}, r)$ this jump is given as the Ecalle acceleration operator
$\mathcal{A}_{1,\frac{k}{k+1},d}$ acting on the hyperfunction $G(s)$. Precisely, we have
\begin{gather*}
  f^{d^+}(t)-f^{d^-}(t)=(\mathcal{A}_{1,\frac{k}{k+1},d}G)(t):=
  G(s)\Big[t^{\frac{-k}{1+k}}C_{\frac{k+1}{k}}((s/t)^{\frac{k}{1+k}})\frac{k}{1+k}s^{-\frac{1}{1+k}}\Big]\\
  = G(s^{\frac{1+k}{k}})\Big[t^{-k/(1+k)}C_{\frac{k+1}{k}}(s/t^{\frac{k}{1+k}})\Big],
\end{gather*}
where 
$G(s)[\varphi(s)]$ is defined by (\ref{eq:kothe}), and the last equality holds by the change of variables, because
 if  $G(s)=[g(s)]_{d}$
then  $G(s^p)=[g(s^p)]_{d/p}$ for every $p>0$.

\section{Characterization of the Stokes phenomenon in a case when the initial data have a finite set of singular points}

In this section we specify a form of the jumps across the Stokes lines based on the solution of the heat equation in a case when the initial data have a finite set of singular points. Due to the linearity of the equation, it is enough to consider the case that the singularity occurs only at one point -- singular or branching point.

Recall the following proposition

\begin{Pro}[{\cite[Theorem 4]{Mic-Pod}}]
\label{th:heat}
Suppose that $\hat{u}$ is a unique formal solution of the Cauchy problem of the heat equation 
\begin{equation}
 \label{eq:heat}
 \begin{cases}
  \partial_t u=\partial_z^2 u\\
  u(0,z)=\varphi(z)
 \end{cases}
\end{equation}
 with
\begin{gather}
\label{eq:heat_cond}
\varphi\in\OO^2 \biggl(D\cup S_{\frac{d}{2}}(\frac{\varepsilon}{2})\cup S_{\frac{d}{2}+\pi}(\frac{\varepsilon}{2})\biggr)\quad
\textrm{for\,\, some}\quad \varepsilon>0.
\end{gather}
Then $\hat{u}$ is $1$-summable in the direction $d$ and for every $\theta\in(d-\frac{{\varepsilon}}{2},d+\frac{{\varepsilon}}{2})$ and
for every $\tilde{\varepsilon}\in(0,\varepsilon)$ there exists
$r>0$ such that  its
$1$-sum $u^{\theta}\in\OO(S_{\theta}(\pi-\tilde{\varepsilon},r)\times D)$ is represented by
\begin{equation}
 \label{eq:heat_solution}
u(t,z)=u^{\theta}(t,z)=
\frac{1}{\sqrt{4\*\pi\*t}}\*\int_{0}^{e^{i\*\frac{\theta}{2}}\*\infty}\,\,\bigl(\varphi(z+s)+\varphi(z-s)\bigr)\,\*e^{\frac{-s^2}{4t}}ds
\end{equation}
for $t\in S_{\theta}( \pi-\tilde{\varepsilon},r)$ and $z\in D_r$.
\end{Pro}

Now consider the heat equation (\ref{eq:heat}) with $\varphi(z)\in\OO^2(\widetilde{{\CC}\setminus\{z_0\}})$ for some $z_0\in\CC\setminus \{0\}$.
First, observe that in this case $\mathcal{L}_{\delta}$ with $\delta:=2\theta:=2\arg z_0$ is a  Stokes line for $\hat{u}$.

For every sufficiently small $\varepsilon>0$ there exists $r>0$ such that for every fixed $z\in D_r$ the jump is given by 
\begin{displaymath}
u^{\delta^+}(t,z) - u^{\delta^-}(t,z)=F_z(s)\biggl[\frac{1}{\sqrt{4\pi t}}e^{-\frac{s^2}{4t}}\biggr],
\end{displaymath}
where $t\in S_{\delta}(\pi - \varepsilon, r)$ and
\begin{multline*}
 F_z(s)=\bigg[\varphi(s+z)+\varphi(z-s)\bigg]_{\theta_z}
 \\=\bigg[\varphi(s+z)\bigg]_{\theta_z} \in\OO^2\bigl(D\cup(S_{\theta}(\alpha)\setminus \mathcal{L}_{\theta_z})\bigr)\diagup
\OO^2\bigl(D\cup S_{\theta}(\alpha)\bigr)
\end{multline*}
and $\theta_z=\arg(z_0-z)$.

\begin{Rem}
In the remainder of this section  we assume that $t\in S_{\delta}(\pi - \varepsilon,r)$ and fixed $z\in D$ ($\varepsilon, r>0$).
\end{Rem}

Now we consider the case when $z_0$ is a single-valued singular point of the function  $\varphi(z)\in\OO^2({{\CC}\setminus\{z_0\}})$.

\begin{Th}
\label{tw:1}
Suppose that $\varphi(z)=\sum_{n=1}^{\infty}\frac{a_n}{(z-z_0)^n}+ \phi(z)$, where $a_1, a_2,...\in\CC$ and $\lim_{n \rightarrow \infty}\sqrt[n]{ |a_n|}<1$, $z_0\in\CC\setminus\{0\}$, $\phi(z)\in\OO^2(\CC)$. Then 
\[
F_z(s)=\biggl[\sum_{n=1}^{\infty}\frac{a_n}{(z+s-z_0)^n}\biggr]_{\theta_z}=-2\pi\*i \sum_{n=1}^{\infty}\frac{a_n (-1)^{n-1}}{(n-1)!}\delta^{(n-1)}(z+s-z_0),
\]
where $\delta$ is the Dirac function and $\delta^{(n-1)}$ denotes its $(n-1)$-th derivative.\\
Moreover, the jump is given by the convergent series
\begin{displaymath}
u^{\delta^+}(t,z) - u^{\delta^-}(t,z)
 =-i\sqrt{\frac{\pi}{t}}\sum_{n=1}^\infty \frac{a_n(-1)^{n-1}}{(n-1)!}\,\,\frac{\mathrm{d}^{n-1}}{\mathrm{d}s^{n-1}}e^{-\frac{s^2}{4t}}\Bigg|_{s=z_0-z}.
 \end{displaymath}
\end{Th}

\begin{proof} Observe that since $\delta(x)=\biggl[-\frac{1}{2\pi i s}\biggr]$ (see \cite{Kaneko}), then $\delta(x-a)=\biggl[-\frac{1}{2\pi i (s-a)}\biggr]$ (where $a\in\RR$) and differentiating it $n$-times one can easily obtain 
\[
\delta^{(n)}(x-a)=\biggl[-\frac{(-1)^n n!}{2\pi i (s-a)^{n+1}}\biggr] \Longrightarrow -\frac{2\pi i\,\delta^{(n-1)}(x-a)}{(-1)^{n-1}(n-1)!}=\biggl[\frac{1}{(s-a)^n}\biggr].
\]
Notice that the same holds for $a=z_0-z\in\CC$.\\
Hence we derive
\[
F_z(s)=\bigg[\sum_{n=1}^\infty \frac{a_n}{(s+z-z_0)^{n}}\bigg]_{\theta_z}=-2\pi i \sum_{n=1}^\infty \frac{a_n (-1)^{n-1}}{(n-1)!}\delta^{(n-1)}(s+z-z_0).
\]
Thus
\begin{multline*}
u^{\delta^+}(t,z) - u^{\delta^-}(t,z)=F_z(s)\biggl[\frac{1}{\sqrt{4\pi t}}e^{-\frac{s^2}{4t}}\biggr]\\
=-2\pi i \sum_{n=1}^\infty \frac{a_n (-1)^{n-1}}{(n-1)!}\delta^{(n-1)}(s+z-z_0)\biggl[\frac{1}{\sqrt{4\pi t}}e^{-\frac{s^2}{4t}}\biggr]\\
=-i\sqrt{\frac{\pi}{t}}\sum_{n=1}^\infty \frac{a_n(-1)^{n-1}}{(n-1)!}\,\,\frac{\mathrm{d}^{n-1}}{\mathrm{d}s^{n-1}}e^{-\frac{s^2}{4t}}\Bigg|_{s=z_0-z}.
\end{multline*}
It remains to prove the convergence of the series above.\\
Notice that by Remark \ref{prop:3} since $s\mapsto  e^{-\frac{s^2}{4t}}\in\OO^2(\CC)$, then there exist $\tilde A, \tilde B>0$ such that for every $t\in S(\theta, \pi-\tilde{\varepsilon},r)$ and $z\in D_r$ (for every sufficiently small $\varepsilon>0$ and $\tilde\varepsilon\in(0,\varepsilon))$ we have that
$$ \Bigg|\delta^{(n-1)}(z+s-z_0)\bigg[e^{-\frac{s^2}{4t}}\bigg]\Bigg|=\Bigg|\frac{\mathrm{d}^{n-1}}{\mathrm{d}s^{n-1}}e^{-\frac{s^2}{4t}}\bigg|_{s=z_0-z}\Bigg|\leq \tilde A \tilde B ^{(n-1)} ((n-1)!)^{\frac{1}{2}},$$
 so 
\begin{multline*} 
\big|u^{\delta^+}(t,z) - u^{\delta^-}(t,z)\big|= \biggl|-i\sqrt{\frac{\pi}{t}}\sum_{n=1}^{\infty}\frac{a_n (-1)^{(n-1)}}{(n-1)!}\delta^{(n-1)}(z+s-z_0)\bigg[e^{-\frac{s^2}{4t}}\bigg]\biggr|\\
 \leq\sqrt{\frac{\pi}{t}}\sum_{n=1}^{\infty}\frac{|a_n|}{(n-1)!}\biggl|\delta^{(n-1)}(z+s-z_0)\bigg[e^{-\frac{s^2}{4t}}\bigg]\biggr|\\
\leq \sqrt{\frac{\pi}{t}}\sum_{n=1}^{\infty}\frac{|a_n|}{(n-1)!}\tilde A \tilde B ^{(n-1)} ((n-1)!)^{\frac{1}{2}}
 = \sqrt{\frac{\pi}{t}}\tilde A\sum_{n=1}^{\infty}\frac{|a_n| \tilde B ^{(n-1)}}{((n-1)!)^{\frac{1}{2}}} <\infty,
  \end{multline*}
  because $\lim_{n \rightarrow \infty}\sqrt[n]{ |a_n|}<1$. Thus this implies the convergence of $u^{\delta^+}(t,z) - u^{\delta^-}(t,z)$.
\end{proof}

In particular, from the above theorem we obtain the following examples.
\begin{Ex}
 Assume now $\varphi(z)=\sum_{n=1}^N\frac{a_n}{(z-z_0)^n}+ \phi(z)$, for some $z_0\in\CC\setminus\{0\}$, where $N\in\NN\setminus\{0\}$, $a_1, a_2,...a_N\in\CC$ and $\phi(z)\in\OO^2(\CC)$. Then 
\[
F_z(s)=\biggl[\sum_{n=1}^{N}\frac{a_n}{(z+s-z_0)^n}\biggr]_{\theta_z}=-2\pi\*i \sum_{n=1}^N \frac{a_n (-1)^{n-1}}{(n-1)!}\delta^{(n-1)}(z+s-z_0),
\]
and the jump is given by
\begin{displaymath}
u^{\delta^+}(t,z) - u^{\delta^-}(t,z)
=-i\sqrt{\frac{\pi}{t}}\sum_{n=1}^N \frac{a_n(-1)^{n-1}}{(n-1)!}\,\,\frac{\mathrm{d}^{n-1}}{\mathrm{d}s^{n-1}}e^{-\frac{s^2}{4t}}\Bigg|_{s=z_0-z}.
 \end{displaymath}
\end{Ex}

\begin{Ex}
Let $\varphi(z)=e^{\frac{1}{z-z_0}}$ for some $z_0\in\CC\setminus\{0\}$. Then 
\[
F_z(s)=\bigg[e^{\frac{1}{z+s-z_0}}\biggr]_{\theta_z}=-2\pi\*i \sum_{k=0}^{\infty}\frac{(-1)^k}{k!(k+1)!}\delta^{(k)}(z+s-z_0),
\]
and the jump is given by
\begin{displaymath}
u^{\delta^+}(t,z) - u^{\delta^-}(t,z)
= -i\sqrt{\frac{\pi}{t}}\sum_{k=0}^{\infty}\frac{(-1)^k}{k!(k+1)!}\,\,\frac{\mathrm{d}^k}{\mathrm{d}s^k} e^{-\frac{s^2}{4t}}\Bigg|_{s=z_0-z}.
\end{displaymath}
\end{Ex}

\bigskip

Let us now consider the general case. For this purpose fix $z\in D$. For $s\in\mathcal L_{\theta_z}$ define (similarly as in \cite{Im} and \cite{S-S}) a function on $\mathcal L_{\theta_z}$ by 
\begin{displaymath}
\mathrm{var}F_z(s)=\left\{\begin{array}{ll}
0 & ,\textrm{if $|s|<|z_0-z|$}\\
\varphi\bigl(z_0+(s+z-z_0)e^{2\pi i}\bigr)- \varphi(s+z) & ,\textrm{if $|s|>|z_0-z|$},
\end{array}\right.
\end{displaymath}
and a Heaviside function in a direction $\theta_z$ by
\begin{displaymath}
H_{\theta_z}(xe^{i\theta_z})=\left\{\begin{array}{ll}
1 & ,\textrm{for $x>0$}\\
0 & ,\textrm{for $x<0$},
\end{array}\right.
\end{displaymath}
thus $F_z(s)=\bigg[\varphi(s+z)\bigg]_{\theta_z}=-\mathrm{var}F_z(s)=-\mathrm{var}F_z(s) H_{\theta_z}(s+z-z_0)$.\\
So $u^{\delta^+}(t,z) - u^{\delta^-}(t,z)=-\mathrm{var}F_z(s)\bigg[ \frac{1}{\sqrt{4\pi t}}\,\,e^{-\frac{s^2}{4t}}\bigg]$ where, in general, $-\mathrm{var}F_z(s)$ is an analytic functional on $\mathcal L_{\theta_z}$.

\emph{Notation}. The set of all measurable functions $f:\mathcal{L}_{\theta_z}\rightarrow\CC$ such that $\int_K |f|dx<\infty$ for all compact sets $K\subset\mathcal{L}_{\theta_z}$ will be denoted by $L^1_{\mathrm{loc}}(\mathcal{L}_{\theta_z}).$

\begin{Th}
\label{tw:3}
Under the above assumptions we have several cases to discuss
\begin{enumerate}
\item $\mathrm{var}F_z(s)\in L^1_{\mathrm{loc}}(\mathcal L_{\theta_z})$ and is an analytic function of exponential growth of order at most 2 for $|s|>|z_0-z|$.\\
Then for every sufficiently small $\varepsilon>0$ there exists $r>0$ such that the jump is given by
\begin{displaymath}
u^{\delta^+}(t,z) - u^{\delta^-}(t,z)=-\int_{z_0-z}^{e^{i\theta_z}\infty} \mathrm{var} F_z(s)\,\,\frac{1}{\sqrt{4\pi t}}\,\,e^{-\frac{s^2}{4t}}\,\,ds,
\end{displaymath}
for $(t,z)\in S_{\delta}( \pi -\varepsilon, r)\times D.$\\

\item $\mathrm{var}F_z(s)$ is a distribution on $\mathcal L_{\theta_z}$ and is an analytic function of exponential growth of order at most 2 for $|s|>|z_0-z|$.\\
 Then there exist $m\in\NN$ and $\mathrm{var}\tilde F_z(s)$ satisfying the assumptions of the case (1) such that $$\frac{\mathrm{d}^m}{\mathrm{d}s^m}\mathrm{var}\tilde F_z(s)=\mathrm{var}F_z(s).$$
Moreover, for every sufficiently small $\varepsilon>0$ there exists $r>0$ such that the jump is given by 
\begin{multline*}
u^{\delta^+}(t,z) - u^{\delta^-}(t,z)=-\mathrm{var}F_z(s)\bigg[\frac{1}{\sqrt{4\pi t}}e^{-\frac{s^2}{4t}}\bigg]\\
=-\frac{\mathrm{d}^m}{\mathrm{d}s^m}\mathrm{var}\tilde F(s) \bigg[\frac{1}{\sqrt{4\pi t}}e^{-\frac{s^2}{4t}}\bigg]=-\mathrm{var}\tilde F(s) \bigg[(-1)^m\frac{\mathrm{d}^m}{\mathrm{d}s^m}\bigg(\frac{1}{\sqrt{4\pi t}}e^{-\frac{s^2}{4t}}\bigg)\bigg]\\
= -\int_{z_0-z}^{e^{i\theta_z}\infty} \mathrm{var}\tilde F_z(s)\,(-1)^m\frac{\mathrm{d}^m}{\mathrm{d}s^m}\bigg(\frac{1}{\sqrt{4\pi t}}\,\,e^{-\frac{s^2}{4t}}\bigg)\,\,ds,
\end{multline*}
for $(t,z)\in S_{\delta}( \pi -\varepsilon, r)\times D.$\\

\item $\mathrm{var}F_z(s)$ is an analytic functional on $\mathcal L_{\theta_z}$.\\
Then $\mathrm{var}F_z(s)=\sum_{n=0}^\infty \mathrm{var}F_{z,n}(s)$, where $\mathrm{var}F_{z,n}(s)$ satisfy the assumptions of the case (2). So for every $n\in\NN$ there exists $k_n\in\NN$ and $\mathrm{var}\tilde F_{z,n}(s)$ satisfying the assumptions of the case (1) such that $$\mathrm{var}F_{z,n}(s)=\frac{\mathrm{d}^{k_n}}{\mathrm{d}s^{k_n}}\mathrm{var}\tilde F_{z,n}(s).$$
Moreover, for every sufficiently small $\varepsilon>0$ there exists $r>0$ such that the jump is given by 
\begin{multline*}
u^{\delta^+}(t,z) - u^{\delta^-}(t,z)=-\mathrm{var}F_z(s)\bigg[\frac{1}{\sqrt{4\pi t}}e^{-\frac{s^2}{4t}}\bigg]
=-\sum_{n=0}^\infty\mathrm{var}F_{z,n}(s)\bigg[\frac{1}{\sqrt{4\pi t}}e^{-\frac{s^2}{4t}}\bigg]\\
=-\sum_{n=0}^\infty \int_{z_0-z}^{e^{i\theta_z}\infty} \mathrm{var}\tilde F_{z,n}(s)\,(-1)^{k_n}\frac{\mathrm{d}^{k_n}}{\mathrm{d}s^{k_n}} \bigg(\frac{1}{\sqrt{4\pi t}}\,\,e^{-\frac{s^2}{4t}}\bigg)ds,
\end{multline*}
for $(t,z)\in S_{\delta}( \pi -\varepsilon, r)\times D.$\\
\end{enumerate}

\end{Th}

\begin{proof}
Ad.(1) First observe that for every $z\in D$ the function $s\mapsto\mathrm{var} F_z(s)$ is analytic on $\mathcal L_{\theta_z}\setminus\{z_0-z\}$, locally integrable and has an exponential growth of order at most 2 as $s\rightarrow \infty, s\in\mathcal L_{\theta_z}$. Hence for every sufficiently small $\varepsilon>0$ there exists $r>0$ such that  the integral $u^{\delta^+}(t,z) - u^{\delta^-}(t,z)$ is well defined for $(t,z)\in S_{\delta}( \pi -\varepsilon, r)\times D.$

For $z_0-z=x_0e^{i\theta_z}$ and $s=xe^{i\theta_z}$, where $x_0, x>0$, we obtain
\begin{multline*}
u^{\delta^+}(t,z) - u^{\delta^-}(t,z)=F_z(s)\biggl[\frac{1}{\sqrt{4\pi t}}e^{-\frac{s^2}{4t}}\biggr]\\
=\frac{1}{\sqrt{4\pi t}} \lim_{\varepsilon\longrightarrow 0^+}\biggl\{ \int_0^\infty \varphi\bigl((x+i\varepsilon)e^{i\theta_z}+z_0 -x_0
e^{i\theta_z}\bigr)e^{-\frac{1}{4t}\big((x+i\varepsilon)e^{i\theta_z}\big)^2}e^{i\theta_z}dx\\
- \int_0^\infty \varphi\bigl((x-i\varepsilon)e^{i\theta_z}+z_0 -x_0 e^{i\theta_z}\bigr)e^{-\frac{1}{4t}\big((x-i\varepsilon)e^{i\theta_z}\big)^2}e^{i\theta_z}dx \biggr\}\\
 =\frac{1}{\sqrt{4\pi t}} \int_0^\infty e^{-\frac{1}{4t}(xe^{i\theta_z})^2}e^{i\theta_z}  \lim_{\varepsilon\longrightarrow 0^+}
 \bigg\{\varphi\bigl((x+i\varepsilon-x_0)e^{i\theta_z}+z_0\bigr)\\
  - \varphi\bigl((x-i\varepsilon-x_0)e^{i\theta_z}+z_0\bigr)\bigg\}\,dx=(*) 
 \end{multline*} 
Observe that
\begin{itemize}
\item for $x - x_0>0$, we have
\begin{multline*} 
\lim_{\varepsilon\longrightarrow 0^+} \bigg\{\varphi\bigl((x+i\varepsilon-x_0)e^{i\theta_z}+z_0\bigr) - \varphi\bigl((x-i\varepsilon-x_0)e^{i\theta_z}+z_0\bigr)\bigg\}\\=\varphi\bigl((x-x_0)e^{i\theta_z}+z_0\bigr)  -\varphi\bigl((x-x_0)e^{i\theta_z}e^{2\pi i}+z_0\bigr)
\end{multline*}

\item for $x - x_0<0$, we have
\begin{multline*}
\lim_{\varepsilon\longrightarrow 0^+} \bigg\{\varphi\bigl((x+i\varepsilon-x_0)e^{i\theta_z}+z_0\bigr) - \varphi\bigl((x-i\varepsilon-x_0)e^{i\theta_z}+z_0\bigr)\bigg\}=0.
\end{multline*}
\end{itemize}
Hence
\begin{multline*}
(*)=\frac{1}{\sqrt{4\pi t}} \int_{x_0}^\infty\bigg\{\varphi\bigl((x-x_0)e^{i\theta_z}+z_0\bigr)  -\varphi\bigl((x-x_0)e^{i\theta_z}e^{2\pi i}+z_0\bigr)\bigg\} \, e^{-\frac{(xe^{i\theta_z})^2}{4t}}e^{i\theta_z} dx\\
= \frac{1}{\sqrt{4\pi t}}\int_{z_0 -z}^{e^{i\theta_z}\infty}\bigg(\varphi(s+z)-\varphi\bigl((s+z-z_0)e^{2\pi i}+z_0\bigr) \bigg)e^{-\frac{s^2}{4t}}ds\\
=-\int_{z_0-z}^{e^{i\theta_z}\infty} \mathrm{var} F_z(s)\,\,\frac{1}{\sqrt{4\pi t}}\,\,e^{-\frac{s^2}{4t}}\,\,ds.
\end{multline*}\\

Ad.(2) Observe that since $\mathrm{var}F_z(s)$ is continuous on $\mathcal L_{\theta_z}\setminus\{z_0-z\}$, by the locally structure theorem for distributions (see Proposition 7.1 \cite{ElKinani} ), there exist $m\in\NN$ and $\mathrm{var}\tilde F_z(s)\in L^1_{\mathrm{loc}}(\mathcal L_{\theta_z})$ such that $$\frac{\mathrm{d}^m}{\mathrm{d}s^m}\mathrm{var}\tilde F_z(s)=\mathrm{var}F_z(s).$$
Furthermore, $\mathrm{var} F_z(s)$ has exponential growth of order at most 2 as $s\rightarrow \infty, s\in\mathcal L_{\theta_z}$, then also $\mathrm{var} \tilde F_z(s)$ has an exponential growth of order at most 2 as $s\rightarrow \infty, s\in\mathcal L_{\theta_z}$. The rest of the proof is analogous to the proof of the case (1).\\

Ad.(3) Notice that since $\mathrm{var}F_{z,n}(s)$ obey the assumptions of the case (2) and based on results in \cite{Kaneko1}
we can write $\mathrm{var}F_z(s)=\sum_{n=0}^\infty \mathrm{var}F_{z,n}(s)$, where $\mathrm{var}F_{z,n}(s)$ satisfy the assumptions of the case (2). Then for every $n\in\NN$ there exists $k_n\in\NN$ and $\mathrm{var}\tilde F_{z,n}(s)$ satisfying the assumptions of the case (1) such that $\mathrm{var}F_{z,n}(s)=\frac{\mathrm{d}^{k_n}}{\mathrm{d}s^{k_n}}\mathrm{var}\tilde F_{z,n}(s).$ The rest of the proof is also similar to the proof of the case (1).
\end{proof}

Now we give two examples of the function $\varphi(z)$ satisfying the case (1) of Theorem \ref{tw:3}.
\begin{Ex}
\label{ex:2}
Assume that $\varphi(z)=\ln(z-z_0)$ for some $z_0\in\CC\setminus\{0\}$. Then 
\[\mathrm{var}F_z(s)=2\pi i H_{\theta_z}(s+z-z_0),\]
and the jump is given by
$$u^{\delta^+}(t,z) - u^{\delta^-}(t,z)=-i\sqrt{\frac{\pi}{t}}\int_{z_0 -z}^{e^{i\theta_z}\infty} e^{-\frac{s^2}{4t}}ds.$$
Indeed, for $|s|>|z_0-z|$ we derive
\begin{multline*}
\mathrm{var} F_z (s)=\varphi\bigl(z_0+(s+z-z_0)e^{2\pi i}\bigr)- \varphi(s+z)=\\
= \ln\biggl(z_0+(s+z-z_0)e^{2\pi i}-z_0\biggr)- \ln\biggl((s+z)-z_0\biggr)=\\
=\ln\bigl((s+z-z_0)e^{2\pi i}\bigr)-\ln(s+z-z_0)=2\pi i.
\end{multline*}
\end{Ex}

\begin{Ex}
\label{ex:3}
Let $\varphi(z)=(z-z_0)^\lambda$ for some $z_0\in\CC\setminus\{0\}$, $\lambda\notin \ZZ$ and $\lambda>-1$.
Then 
\[\mathrm{var} F_z(s)=2i H_{\theta_z}(z+s-z_0) (-s-z+z_0)^\lambda \sin(\lambda\pi),\]
and the jump is given by
\[
u^{\delta^+}(t,z) - u^{\delta^-}(t,z)=-\frac{i}{\sqrt{\pi t}}  \int_{z_0 -z}^{e^{i\theta_z}\infty} e^{-\frac{s^2}{4t}}(-s-z+z_0)^\lambda \sin(\lambda\pi)ds.
\]

 More precisely, for $|s|>|z_0-z|$
\begin{multline*}
\mathrm{var} F_z (s)=\varphi\bigl(z_0+(z+s-z_0)e^{2\pi i}\bigr)- \varphi(z+s)\\
= \biggl(\bigl(z_0+(z+s-z_0)e^{2\pi i}\bigr)-z_0\biggr)^{\lambda}- \biggl((z+s)-z_0\biggr)^{\lambda}\\
=\bigl((z+s-z_0)e^{2\pi i}\bigr)^{\lambda}-(z+s-z_0)^{\lambda}=(z+s-z_0)^{\lambda} (e^{2\pi i\lambda} -1)\\
=2i(-1)^{\lambda}(z+s-z_0)^{\lambda}\sin(\pi\lambda),
\end{multline*}
because 
\begin{multline*}
\sin(\pi\lambda)=\frac{e^{i\pi\lambda}-e^{-i\pi\lambda}}{2i}=\frac{e^{2i\pi\lambda}-1}{2ie^{i\pi\lambda}}\\
 \Longrightarrow e^{2i\pi\lambda}-1=2ie^{i\pi\lambda}\sin(\pi\lambda)=2i(-1)^{\lambda}\sin(\pi\lambda).
\end{multline*}
\end{Ex}

\bigskip

Now we present an example of the function $\varphi(z)$ satisfying the case (2) of Theorem \ref{tw:3}.
\begin{Ex}
Let again $\varphi(z)=(z-z_0)^\lambda$ for some $z_0\in\CC\setminus\{0\}$, $\lambda\notin \ZZ$ and $\lambda <-1$.
Then for $m=\lfloor -\lambda \rfloor$ we can define $\mathrm{var}\tilde {F}_z (s)\in L^1_{\mathrm{loc}}(\mathcal L_{\theta_z})$ by
\[
\mathrm{var}\tilde {F}_z (s)=\frac{2i(-1)^{\lambda+m}\sin(\pi(\lambda+m))}{(\lambda+1)(\lambda+2)\dots(\lambda+m)}(s+z-z_0)^{\lambda+m}H_{\theta_z}(s+z-z_0),
\]
thus
\begin{multline*}
\mathrm{var}F_z(s)=\frac{\mathrm{d}^m}{\mathrm{d}s^m}\mathrm{var}\tilde {F}_z (s)\\
= \frac{\mathrm{d}^m}{\mathrm{d}s^m}\bigg\{\frac{2i(-1)^{\lambda+m}\sin(\pi(\lambda+m))}{(\lambda+1)(\lambda+2)\dots(\lambda+m)}(s+z-z_0)^{\lambda+m}H_{\theta_z}(s+z-z_0)\bigg\},
\end{multline*}
and the jump is given by
\begin{multline*}
u^{\delta^+}(t,z) - u^{\delta^-}(t,z)\\
=\frac{-i}{\sqrt{\pi t}}\int_{z_0-z}^{e^{i\theta_z}\infty} \frac{(-1)^\lambda (s+z-z_0)^{\lambda+m}\sin((\lambda+m)\pi)}{(\lambda+1)(\lambda+2)\dots (\lambda+m)}\frac{\mathrm{d}^m}{\mathrm{d}s^m}\bigg(e^{-\frac{s^2}{4t}}\bigg)ds.
\end{multline*}
\end{Ex}

\bigskip

Finally, we give an example of the function $\varphi(z)$ that satisfies the case (3) of Theorem \ref{tw:3}.
\begin{Ex} 
\label{ex:4}
Assume now that $\varphi(z)=e^{\frac{1}{(z-z_0)^{\lambda}}}$ where $z_0\in\CC\setminus\{0\}$, $\lambda\notin\QQ$ and $\lambda>0$. Then for $k_n=\lfloor \lambda n \rfloor$ we can define functions $\mathrm{var}\tilde F_{z,n}(s)\in L_{\mathrm{loc}}^1(\mathcal L_{\theta_z})$ by
\[
\mathrm{var}\tilde F_{z,n}(s)=\frac{2i(-s-z+z_0)^{-\lambda n+k_n}\sin((-\lambda n +k_n)\pi)}{n! (-\lambda n+1)(-\lambda n+2)\dots (-\lambda n +k_n)}H_{\theta_z}(s+z-z_0),
\]
and
\begin{multline*}
\mathrm{var}F_z(s)=\sum_{n=0}^{\infty}\mathrm{var}F_{z,n}(s)=\sum_{n=0}^{\infty}\frac{\mathrm d^{k_n}}{\mathrm d s^{k_n}}\mathrm{var}\tilde F_{z,n}(s)\\
=\bigg(\sum_{n=0}^\infty \frac{\mathrm d^{k_n}}{\mathrm ds^{k_n}}\frac{2i(-s-z+z_0)^{-\lambda n+k_n}\sin((-\lambda n +k_n)\pi)}{n! (-\lambda n+1)(-\lambda n+2)\dots (-\lambda n +k_n)}\bigg)H_{\theta_z}(s+z-z_0).
\end{multline*}
Then the jump is given by
\begin{multline*}
u^{\delta^+}(t,z) - u^{\delta^-}(t,z)=\\
-\int_{z_0-z}^{e^{i\theta_z}\infty}\sum_{n=0}^\infty \frac{i(-s-z+z_0)^{-\lambda n+k_n}\sin((-\lambda n +k_n)\pi)}{\sqrt{\pi t}n! (-\lambda n+1)(-\lambda n+2)\dots (-\lambda n +k_n)}\biggl[(-1)^{k_n}\frac{\mathrm d^{k_n}}{\mathrm ds^{k_n}}\bigg(e^{-\frac{s^2}{4t}}\bigg)\biggr]ds.\\
\end{multline*}
Observe that by Remark \ref{prop:3} since $s\mapsto e^{-\frac{s^2}{4t}}\in\OO^2(\CC)$, then there exist $ A, B>0$ such that for every $t\in S(\theta, \pi-\tilde{\varepsilon},r)$ and $z\in D_r$ (for every sufficiently small $\varepsilon>0$ and $\tilde\varepsilon\in(0,\varepsilon))$ we have that
\\ \bigg|$(-1)^{k_n}\frac{\mathrm d^{k_n}}{\mathrm ds^{k_n}}\bigg(e^{-\frac{s^2}{4t}}\bigg)\bigg|\leq A B ^{\lambda n} (n!)^{\frac{\lambda}{2}}$ 
and 
$$\bigg|\frac{1}{n! (-\lambda n+1)(-\lambda n+2)\dots (-\lambda n +k_n)}\bigg|\leq\frac{1}{(n!)^{\frac{\lambda}{2}+1}}<\infty$$
hence analogously to the proof of Theorem \ref{tw:1} we obtain the convergence of the above series.
\end{Ex}

\bigskip
At the end of this section, we similarly derive jumps for the following generalization of the heat equation
 \begin{equation}
 \label{eq:generalization}
 \begin{cases}
\partial_t^p u(t,z)=\partial_z^q u(t,z),\,\, p,q\in\NN,\,\,1\leq p<q\\
u(0,z)=\varphi(z) \in\OO(D)\\
\partial_{t}^{j}u(0,z)=0\ \textrm{for}\,\, j=1,2,\dots,p-1,
 \end{cases}
\end{equation}
with $\varphi(z)\in\OO^{\frac{q}{q-p}}\Bigl(D\cup \bigcup_{l=0}^{q-1} S_{\frac{dp}{q}+\frac{2\pi\*l}{q}}(\frac{\varepsilon p}{q})\Bigr)$ for some $\varepsilon>0$.

Then a unique formal solution $\hat{u}(t,z)$ of this Cauchy problem is $\frac{p}{q-p}$-summable in the direction $d$ and for every
$\psi\in(d-\frac{{\varepsilon}}{2},d+\frac{{\varepsilon}}{2})$ and for every $\tilde{\varepsilon}\in(0,\varepsilon)$ there exists $r>0$ such that
its $\frac{p}{q-p}$-sum $u\in\OO(S_d(\frac{\pi(q-p)}{p}-\tilde{\varepsilon},r)\times D)$ is given by (see \cite[Theorem 6]{Mic-Pod})
\begin{multline}
 \label{eq:generalization_solution}
u(t,z)=u^{\psi}(t,z)=\frac{1}{q\*\sqrt[q]{t^p}}\*\int_{0}^{e^{\frac{i\psi p}{q}}\infty}\bigl(\varphi(z+s)+\dots+\varphi(z+e^{\frac{2(q-1)\pi\*i}{q}}\*s)\bigr)C_{\frac{q}{p}}(\frac{s}{\sqrt[q]{t^p}})ds.
\end{multline}
\bigskip
As in the case of the heat equation (\ref{eq:heat}), we assume that $\varphi(z)\in\OO^{\frac{q}{q-p}}(\widetilde{{\CC}\setminus\{z_0\}})$.

Then $\mathcal{L}_{\delta}$ with $\delta:=q\theta/p:=q\arg z_0/p$ is a separate Stokes line for $\hat{u}$,
such that $\delta_z=q\arg(z_0-z)/p$ for every sufficiently small $z$. 

For every sufficiently small $\varepsilon>0$ there exists $r>0$ such that for every fixed $z\in D_r$ the jump is given by 
\begin{multline*}
u^{\delta^+}(t,z) - u^{\delta^-}(t,z)=F_z(s)\biggl[\frac{1}{q\*\sqrt[q]{t^p}}C_{\frac{q}{p}}(s/\sqrt[q]{t^p})\biggr]\\
=\biggl[\varphi(z+s)+\dots+\varphi(z+e^{\frac{2(q-1)\pi\*i}{q}}\*s) \biggr]_{\theta_z}\biggl[\frac{1}{q\*\sqrt[q]{t^p}}C_{\frac{q}{p}}(s/\sqrt[q]{t^p})\biggr]\\
=\biggl[\varphi(z+s)\biggr]_{\theta_z}\biggl[\frac{1}{q\*\sqrt[q]{t^p}}C_{\frac{q}{p}}(s/\sqrt[q]{t^p})\biggr],
\end{multline*}
(the last equality arising from the fact that in this case all singular points appear in the function $\varphi(z+s)$), where the hyperfunction 
$F_z(s)=\bigg[\varphi(z+s)+\dots+\varphi(z+e^{\frac{2(q-1)\pi\*i}{q}}\*s)\bigg]_{\theta_z}$ belongs to the space 
$\OO^{\frac{q}{q-p}}\bigl(D\cup(S_{\theta}(\alpha)\setminus \mathcal{L}_{\theta_z})\bigr)\diagup
\OO^{\frac{q}{q-p}}\bigl(D\cup S_{\theta}(\alpha)\bigr)$
with $\theta_z=\arg(z_0-z)$.\\

Thus, we obtain analogous results as for the heat equation (\ref{eq:heat}), only that in Theroem \ref{tw:1} and Theorem \ref{tw:3} we replace $\frac{1}{\sqrt{4\pi t}}\,\,e^{-\frac{s^2}{4t}}$ by $\frac{1}{q\*\sqrt[q]{t^p}}C_{\frac{q}{p}}(s/\sqrt[q]{t^p})$.

\bigskip

\textbf{Acknowlegments}\\
The author would like to thank the anonymous referee for valuable comments and suggestions.

\bibliographystyle{siam}
\bibliography{summa}
\end{document}